\documentclass[11 pt]{article}
\usepackage[top=1in, bottom=1in, left=1in, right=1in]{geometry} 
\usepackage{amsmath}
\usepackage{amsthm}
\usepackage{amsfonts}
\usepackage{amssymb}
\usepackage{mathrsfs}
\usepackage{url}
\usepackage{enumerate}
\usepackage[shortlabels]{enumitem}
\usepackage{multirow}
\usepackage{float}
\usepackage{mathtools}
\DeclareMathOperator{\Sym}{Sym}

\DeclareMathOperator{\Frob}{Frob}

\DeclareMathOperator{\GSp}{GSp}
\DeclareMathOperator{\USp}{USp}
\DeclareMathOperator{\End}{End}
\DeclareMathOperator{\Gal}{Gal}
\DeclareMathOperator{\SU}{SU}
\DeclareMathOperator{\U}{U}
\DeclareMathOperator{\Ind}{Ind}
\DeclareMathOperator{\GL}{GL}
\DeclareMathOperator{\SL}{SL}
\DeclareMathOperator{\Sp}{Sp}

\DeclareMathOperator{\Tr}{Tr}
\DeclareMathOperator{\Id}{Id}
\DeclareMathOperator{\SO}{SO}
\DeclareMathOperator{\OO}{O}
\newcommand{\R}{\mathbb{R}}
\newcommand{\C}{\mathbb{C}}
\newcommand{\Z}{\mathbb{Z}}
\newcommand{\Q}{\mathbb{Q}}

\newcommand{\tensor}{\otimes}

\newtheorem{innercustomgeneric}{\customgenericname}
\providecommand{\customgenericname}{}
\newcommand{\newcustomtheorem}[2]{%
  \newenvironment{#1}[1]
  {%
   \renewcommand\customgenericname{#2}%
   \renewcommand\theinnercustomgeneric{##1}%
   \innercustomgeneric
  }
  {\endinnercustomgeneric}
}

\newcustomtheorem{customthm}{Theorem}
\newcustomtheorem{customlemma}{Lemma}

\newtheorem{thm}{Theorem}[section]

\newtheorem{conj}[thm]{Conjecture}

\theoremstyle{definition}
\newtheorem{defn}[thm]{Definition}

\theoremstyle{remark}
\newtheorem{rmk}[thm]{Remark}
\title{Sato-Tate Distributions on Abelian Surfaces}
\author{Noah Taylor\footnote{The author is supported in part by NSF Grant DMS-$1701703$.}}
\date{July 2, 2019}
\begin{document}
\maketitle
\begin{abstract}We prove a few new cases of the Sato-Tate conjecture, using a new automorphy theorem of Allen et al.  Then in the unproven cases, we use partial results to describe nontrivial asymptotics on the trace of Frobenius, and prove their optimality given current knowledge.\end{abstract}
\section{Introduction}
Let $C$ be a genus $g$ curve over a number field $F$.  Given a prime $v$ of $F$, with residue field $\mathbb{F}_v$ of size $q_v$, a theorem of Hasse says that the number $N_v$ of $\mathbb{F}_v$ points on $C$ is between $q_v+1-2g\sqrt{q_v}$ and $q_v+1+2g\sqrt{q_v}$, so that \[a_v:=\frac{q_v+1-N_v}{\sqrt{q_v}}\in[-2g, 2g].\]The Sato-Tate conjecture asks for the distribution of the $a_v$ in $[-2g, 2g]$ as $q_v\rightarrow\infty$, and predicts that they are equidistributed (after passing to a finite extension $F'/F$) with respect to a measure depending on the Mumford-Tate group of the Jacobian of $C$.  For example, if $E$ is an elliptic curve with CM, the distribution is given either by the pushforward of the Haar measure of $\SO(2)$ or of $\OO(2)$ under the trace map.  It has also been proven in \cite{MR2630056} and \cite{MR2827723} that if $F$ is totally real and $E$ does not have CM, then the distribution is the pushforward of the Haar measure of $\SU(2)$.

We look at genus $g=2$ curves and $2$-dimensional abelian surfaces.  In complete analogy with the elliptic curve case, \cite{MR2982436} describes $52$ possible subgroups of $\USp(4)$ whose pushforwards describe the normalized point counts $a_v$ for a genus $2$ curve, and notes that it is likely possible to prove the Sato-Tate conjecture in many cases with a similar method to that of the elliptic curve case.  \cite{MR3660222} uses the powerful potential automorphy theorem of \cite{MR3152941} to prove the conjecture for all but five of the non-generic cases that occur over totally real fields.  In this paper we will use a more powerful potential automorphy theorem of \cite{ACCGHLNSTT} to extend the proof in \cite{MR3660222}, and then we extend \cite{MR3660222}'s work to prove the conjecture for four other subgroups.  Of course, given the Jacobian $J(C)$ of a genus $2$ curve $C$, we can obtain the numbers $a_v$ directly from $J(C)$, by taking the normalized trace of the action of $\Frob_v$, so we may forget about the curve $C$ entirely and work directly with abelian surfaces.

The theorems we prove are as follows:

\begin{customthm}{3.4}If $A/F$ is an abelian surface, $F$ a totally real field, which has a two-dimensional real endomorphism ring defined over a quadratic extension of $F$ which is either totally real or CM, then the Sato-Tate conjecture holds for $A$.\end{customthm}
\begin{customthm}{3.6}If $A/F$ is a (not necessarily simple) abelian surface, $F$ a totally real field, which has quaternionic multiplication defined over a dihedral extension, then the Sato-Tate conjecture holds for $A$.\end{customthm}

These two theorems are equidistribution results, so we know the exact distributions of the $a_v$.  However, we cannot currently prove the Sato-Tate conjecture for $A$ if the endomorphism ring of $A$ is $\Z$, or if the quadratic extension described in Theorem $3.4$ is neither totally real or CM.  In these cases, we prove lesser results:

\begin{customthm}{4.1}If $A/F$ is an abelian surface, $F$ a totally real field, then for any $\varepsilon>0$, $a_v<-\dfrac{2}{3}+\varepsilon$ for a positive proportion of primes $v$, and $a_v>\dfrac{2}{3}-\varepsilon$ for a positive proportion of primes $v$.\end{customthm}
\begin{customthm}{4.3}If $A/F$ is an abelian surface over a totally real field which has a two-dimensional real endomorphism ring defined over a quadratic extension of $F$, then $a_v<-2.47$ for a positive proportion of primes and $a_v>2.47$ for a positive proportion of primes.\end{customthm}

The paper is divided as follows: In section $2$, we set up the terminology and state the Sato-Tate conjecture precisely.  Section $3$ is devoted to proving Theorems $3.4$ and $3.6$ above, and the goal of section $4$ is to prove the asymptotics in Theorems 4.1 and 4.3, as well as others in Theorems 4.2 and 4.4.  The appendix in section $5$ explains the best possible theorems we can obtain in section $4$.
\section{Setup}
\subsection{The Conjecture}
To set up the Sato-Tate conjecture, we follow \cite[Section 2]{MR2982436}.  Fix a number field $F$, an embedding into $\overline{\Q}$, and an embedding of $\overline{\Q}$ into $\C$.  Let $A$ be an abelian variety of dimension $2$ over $F$.  We choose a polarization of $A$.  Given a prime $\ell$, this allows the identification of the $\ell$-adic Tate module with the etale and singular homologies \[V_{\ell}(A)\simeq H_{1, \text{et}}(A_{\overline{\Q}},\Q_\ell)\simeq H_{1, \text{et}}(A_{\C},\Q_\ell)\simeq H_1(A^{\text{top}}_{\C},\Q_\ell)\simeq H_1(A^{\text{top}}_{\C},\Q)\otimes_{\Q}\Q_{\ell}.\]The Weil pairing on the dual of the Tate module $\widehat{V_{\ell}(A)}$ corresponds to the cup product pairing on the cohomologies, so it is a nondegenerate alternating pairing and, given a symplectic basis of $\widehat{V_{\ell}(A)}$, induces a continuous map $\rho_{A, \ell}: G_F\rightarrow \GSp_4(\overline{\Q_{\ell}})$.  We let $G_\ell$ be the image of this map, and $G_\ell^{\text{Zar}}$ be the Zariski closure in $\GSp_4(\overline{\Q_\ell})$.  Then we let $G_F^1$ be the kernel of the cyclotomic character $\chi_\ell:G_F\rightarrow \mathbb{Z}_\ell^{\times}$, so that $g\in G_F^1$ acts trivially on the Weil pairing.  Then $G_\ell^1$ is the image of $G_F^1$ under $\rho_{A, \ell}$ and $G_\ell^{1, \text{Zar}}$ is the Zariski closure.  Because $G_F^1$ acts trivially on the Weil pairing, reconsidering it as a pairing on the vector space, $G_\ell^{1, \text{Zar}}$ is the kernel of the similitude character \[\psi: G_\ell^{\text{Zar}}\rightarrow \mathbb{Z}_\ell^{\times}, \langle h v, hw\rangle=\psi(h)\langle v, w\rangle.\]

Fix an isomorphism $\iota: \overline{\Q_\ell}\rightarrow\C$ for this $\ell$.  We then define $G=G_\ell^{\text{Zar}}\otimes_{\overline{\Q_\ell}}\C$ and $G^1=G_\ell^{1,\text{Zar}}\otimes_{\overline{\Q_\ell}}\C$; then $G/G^1\simeq\C$ via the similitude character.  We look at the image of $\Frob_v$ in this quotient for $v$ a prime of $F$ with residue field $\mathbb{F}_{q_v}$.  Certainly $\Frob_v(\zeta_{\ell^n})=\zeta_{\ell^n}^{q_v}$ so $\Frob_v$ maps to $q_v$.  An argument of Deligne, summarized in \cite[Section~8.3.2]{MR2920749}, shows that the center of the original $\GSp(4)$ lies in the center of $G$, so we may divide $\rho_{A, \ell}(\Frob_v)$ by $q_v^{\frac{1}{2}}$ to get an element $g_v$ in $G^1$ whose eigenvalues have norm $1$ because of the Weil conjectures.

\begin{defn}The Sato-Tate group $ST_A$ of $A$ is a maximal compact Lie subgroup of $G^1$ inside $\USp(4)$, which depends on $\ell$ and the embedding $\iota$.\end{defn}

The element $g_v$ has eigenvalues of norm $1$ so its semisimple component (and even itself, because as described in the errata to \cite{MR2982436}, $g_v$ is already semisimple) lies in some conjugate of $ST_A$; we let $s(v)$ denote its conjugacy class.  The Sato-Tate conjecture is as follows:

\begin{conj}The elements $s(v)$ are equidistributed among the conjugacy classes of $ST_A$, under the pushforward of the Haar measure from $ST_A$.\end{conj}

We record that the Sato-Tate group has a common model over $\Q$ over all $\ell$, as in \cite[Theorem~2.16]{MR2982436}, but it's not known whether the conjugacy classes $s(v)$ themselves are independent of $\ell$.

\subsection{Proof strategy}
Suppose $S$ is the set of primes outside of which $\rho_{A, \ell}$ is unramified.  The general idea for proof is laid out in \cite{MR1484415}; therein the following theorem is shown.
\begin{thm}Suppose that, for any irreducible representation $r$ of $ST_A$, the $L$-function \[L^S(r, s)=\prod_{v\not\in S}\frac{1}{\det(1-r(s(v))q_v^{-s})}\]has a meromorphic extension to the half-plane $\text{Re}(s)\geq 1$, with no poles or zeroes except possibly at $s=1$.  Then the elements $s(v)$ are equidistributed in the conjugacy classes of $ST_A$ if and only if the $L$-functions $L^S(r, s)$ for irreducible nontrivial $r$ have no zero or pole at $s=1$.\end{thm}

We denote the property of having no zeroes or poles on a region invertibility.  The $L$-function has factors at primes of $S$ as well, but their factors do not add poles or zeroes so we ignore them.  To show invertibility of these $L$-functions, the only known method is to equate them to $L$-functions of automorphic forms, a la \cite{MR1954941}, \cite{MR2630056}.  \cite{MR3660222} covers most cases using \cite[Theorem~5.4.1]{MR3152941}; we introduce a new more widely applicable theorem of \cite{ACCGHLNSTT}.  We refer to \cite[Section~5.1]{MR3152941} for the definition of a weakly compatible system.

\begin{defn}A weakly compatible system of representations of $G_F$ is a $5$-tuple \\ $(M, S, \{Q_v(x)\}, \{r_{\lambda}\}, \{H_{\tau}\})$ with $S$ a finite set of $F$-primes satisfying\begin{itemize}
\item $M$ is a number field, and $\{r_{\lambda}\}$ is a set of representations of $G_F$ each defined over $M$, indexed over the primes $\lambda$ of $M$.  If $v\not\in S$ is a prime of $F$, then for $\lambda$ not over the same rational prime $p$ as $v$, $r_{\lambda}$ is unramified at $v$.
\item The polynomials $Q_v(x)$ have rational coefficients and the characteristic polynomial of $r_{\lambda}(\Frob_v)$ is equal to $Q_v(x)$, independent of $\lambda$.
\item If $v$ and $\lambda$ are over the same rational prime $p$, then $r_{\lambda}$ is de Rham at $v$; furthermore, if $v\not\in S$, then $r_{\lambda}$ is crystalline at $v$.
\item For each embedding $\tau: F\xhookrightarrow{} \overline{M}$, the Hodge-Tate weights of $r_{\lambda}$ are given by the multiset $H_{\tau}$, and are in fact independent of $\lambda$.
\end{itemize}
\end{defn}

\begin{thm}[{\cite[Corollary~7.1.11]{ACCGHLNSTT}}]Suppose that $F$ is a CM field and that the $5$-tuple \\ $\mathcal{R}=(M, S, \{Q_v(x)\}, \{r_{\lambda}\}, \{H_{\tau}\})$ is a rank $2$ weakly compatible system of $l$-adic representations of $G_F$ such that $H_{\tau}=\{0, 1\}$ for all $\tau$ and such that $\mathcal{R}$ is strongly irreducible.  If $m$ is a nonnegative integer, then there exists a finite CM extension $F_m/F$ with $F_m/\Q$ Galois such that the weakly compatible system $\text{Symm}^m\mathcal{R}|_{G_{F_m}}$ is automorphic.\end{thm}  We recall that a strongly irreducible system is one where each representation is irreducible even after restricting to finite-index subgroups of $G_F$.

\begin{rmk}The difference between this theorem and \cite[Theorem~5.4.1]{MR3152941} that we take advantage of is that \cite[Theorem~5.4.1]{MR3152941} requires all towers to be either CM or totally real.  In contrast, \cite[Corollary~7.1.11]{ACCGHLNSTT} allows us to base-change from our totally real field $F$ to a CM field $F'$, find an extension $F_m$ over which the compatible systems $\text{Symm}^m\mathcal{R}|_{G_{F_m}}$ are automorphic, and be allowed the added condition that $F_m/F$ is Galois.  This is not possible with the theorem of \cite{MR3152941}; in asking that $F_m/F$ be Galois, we are only allowed base-change to totally real $F'$.\end{rmk}
\section{Sato-Tate for certain $ST_A$}
We introduce the cases of the Sato-Tate conjecture we will prove.  Let $A$ be an abelian surface defined over a field $F$.  If $L$ is the smallest field over which all endomorphisms of $A$ are defined, we define the Galois type of $A$ to be the pair $(\End_L(A)\tensor\R, \Gal(L/F))$ of a real algebra and a group with an action on the algebra.  \cite[Theorem 4.3]{MR2982436} proves that there is a correspondence between the Sato-Tate group and the Galois type of an abelian surface with the following property: if the type $(E, G)$ corresponds to the Sato-Tate group $K$, then the algebra $E$ corresponds bijectively to the identity component $K_0$ of $K$, and $G$ is isomorphic to the component group $K/K_0$.

Therefore, we can equivalently divide the conjecture into cases indexed by the connected component of the Sato-Tate group or by the endomorphism algebra $\End_L(A)\tensor\R$, which can then be further subdivided by including the component group.  There are $6$ possible endomorphism algebras laid out in \cite[Theorem 4.3]{MR2982436} listed below, along with the corresponding Sato-Tate connected component and its embedding into $\USp(4)$:
\begin{itemize}\item\textbf{A}: $\End_L(A)\tensor\R=\R$, corresponding to $ST^0_A=\USp(4)$
\item\textbf{B}: $\End_L(A)\tensor\R=\R\times\R$, corresponding to $ST^0_A=\SU(2)\times\SU(2)$ via $M_1\times M_2\rightarrow\begin{pmatrix}M_1&0\\0&M_2\end{pmatrix}$.
\item\textbf{C}: $\End_L(A)\tensor\R=\R\times\C$, corresponding to $ST^0_A=\SU(2)\times\U(1)$ via $M\times z\rightarrow\left(\begin{smallmatrix}M&\,&\,\\ \,&z&\, \\ \,&\,&\overline{z}\end{smallmatrix}\right)$
\item\textbf{D}: $\End_L(A)\tensor\R=\C\times\C$, corresponding to $ST^0_A=\U(1)\times\U(1)$ via $z\times w\rightarrow \left(\begin{smallmatrix}z&\,&\,&\,\\ \,&\overline{z}&\,&\, \\ \,&\,&w&\, \\ \,&\,&\,&\overline{w}\end{smallmatrix}\right)$
\item\textbf{E}: $\End_L(A)\tensor\R=M_2(\R)$, corresponding to $ST^0_A=\SU(2)$ via $M\rightarrow\begin{pmatrix}M&0\\0&\overline{M}\end{pmatrix}$
\item\textbf{F}: $\End_L(A)\tensor\R=M_2(\C)$, corresponding to $ST^0_A=\U(1)$ via $z\rightarrow\begin{pmatrix}z\cdot I_2&0\\0&\overline{z}\cdot I_2\end{pmatrix}$
\end{itemize}
Further subdividing this list, we obtain 52 distinct Galois types, corresponding to 52 distinct Sato-Tate groups.  Of these, 35 arise as the Sato-Tate group of an abelian surface defined over a totally real field, and 34 of those arise from an abelian surface defined over $\Q$.  Almost nothing is known about the single group of type \textbf{A}; in \cite{MR3660222}, the Sato-Tate conjecture was fully proven for all groups of types \textbf{D} and \textbf{F}, for all totally real abelian surfaces giving rise to groups of type \textbf{C}, and for all totally real abelian surfaces giving rise to one of two groups of type \textbf{B} and six of ten groups of type \textbf{E}.  In addition, assuming that $L$ was also totally real, all other cases were proven.  We describe the remaining cases and prove them with a weakened hypothesis on $L$.

\subsection{Preliminaries}
Before we discuss specific Sato-Tate groups, let us recall standard facts about Galois representations coming from the abelian varieties we study.
\begin{defn}Suppose $A$ is an abelian variety defined over $F$.  We say $A$ is of $\GL_2$-type if it is isogenous over $F$ to a product $A_1\times A_2\times \ldots A_k$ of simple abelian varieties, each also defined over $F$, and with a field $K_i\xhookrightarrow{}\End_F(A_i)\tensor\Q$ with $[K_i:\Q]=\dim(A_i)$.\end{defn}
Given a simple abelian surface $A/F$ of $\GL_2$-type with field $K$ and a rational prime $\ell$, the dual of the $\ell$-adic Tate module $T_\ell$ gives rise to an $\ell$-adic Galois representation $G_F\rightarrow\GL_4(\Q_{\ell})$, isomorphic to the $\ell$-adic etale cohomology of $A$.  The image lands in $\GL_2(\Q_{\ell}\tensor K)$.  For each embedding $\lambda: K\rightarrow\overline{\Q_{\ell}}$, we get a map from this image to $\GL_2(K_\lambda)$ for $K_\lambda$ the completion of $K$ at $\lambda$.  Thus for each embedding of $K$ into $\overline{\Q_\ell}$ for each $\ell$ we obtain a representation $\rho_{A,\lambda}:G_F\rightarrow\GL_2(K_{\lambda})$.  These form a weakly compatible system $(\rho_{A, \lambda})_{\lambda}$.

\begin{thm}[{\cite[Theorems 3.1, 3.2]{MR1212980}}]The weakly compatible system $(\rho_{A, \lambda})_{\lambda}$ is regular of Hodge-Tate weights $0$ and $1$, totally odd and pure of weight $1$.  If $K$ is a real quadratic field, then $\det\rho_{A, \lambda}=\chi_{\ell}$, the $\ell$-adic cyclotomic character; if $K$ is imaginary quadratic, then $\det\rho_{A, \lambda}=\epsilon\tensor\chi_{\ell}$ for some finite-image character $\epsilon$ independent of $\ell$.\end{thm}

In each case below, we will consider the irreducible representations of the Sato-Tate group.  We will extend these in a natural way to representations of $G^1$.  These will be algebraic representations of $G^1$, so that we get compatible systems of representations of $G^{1,\text{Zar}}_{\ell}$.  We can then obtain representations of $G^{\text{Zar}}_{\ell}$ by extending to the central $\mathbb{G}_m$.  Finally obtaining this, we get a compatible system of representations of the Galois group $G_F$, and we can thus use Theorem 2.5 above, combined with Rankin-Selberg theory, to show that the original $L$-function is invertible, as required.  This method will be detailed further in the subsections below.

\subsection{$\textbf{B}[C_2]$}
When we discuss $\textbf{B}[C_2]$, the Sato-Tate group is $\langle\SU(2)\times\SU(2),J\rangle$ where $J=\left(\begin{smallmatrix}\,&\,&\,&1\\ \,&\,&\text{-1}&\,\\ \,&\text{-1}&\,&\, \\ 1&\,&\,&\,\end{smallmatrix}\right)$.  This corresponds to either the case where $A$ is isogenous to a direct sum of nonisogenous elliptic curves, each without CM, or when $A$ is simple but has multiplication by a real quadratic field.  In these cases, $\Q\tensor\End_{\overline{\Q}}(A)$ is either $\Q\times\Q$ or real quadratic.  Conjecture 2.2 in the first case has been proven as \cite[Theorem 5.4]{MR2641185} assuming a few ``Expected Theorems''.  These have been proven since the writing of the paper; see \cite{MR2827723} for a discussion.  We henceforth assume $\Q\tensor\End_{\overline{\Q}}(A)=K$ is a real quadratic field.  Because we're in the $\textbf{B}[C_2]$ case, $A$ is not of $\GL_2$ type over $F$, but is of $\GL_2$ type over a quadratic extension.

We look first at representations of $ST^0_A=\SU(2)\times\SU(2)$ which is an index $2$ subgroup of $ST_A$.  The irreducible representations of $\SU(2)$ are $\Sym^k(St)$ for $St$ the standard $2$-dimensional representation and $k\geq 0$; hence the irreducible representations of $\SU(2)\times\SU(2)$ are $r_{k, l}=\Sym^k(St)\tensor\Sym^{l}(St)$ for $k, l\geq 0$.  We deduce the representations of $ST_A$ using the following standard theorem of Clifford theory (in this form found as \cite[Lemma 23]{MR3660222}, the proof being the author's own):
\begin{thm}If $H\leq G$ is an index $2$ subgroup, and $r$ is a finite-dimensional irreducible representation of $H$, then $r$ extends to a representation of $G$ if and only if $r$ is isomorphic to $r^x$, where $r^x$ is the representation of $H$ defined as $r^x(h)=r(xhx^{-1})$ for $x\in G\backslash H$.  If this is the case, then $r$ extends to exactly two nonisomorphic irreducible representations $r_0$ and $r_0\tensor\chi$ for $\chi$ the nontrivial character $G/H\rightarrow\{\pm 1\}$.  The irreducible representations are exactly those arising from such $r$, along with the inductions $\Ind_H^G\rho$ of all representations $\rho$ of $H$ that do not satisfy the above property.\end{thm}
\begin{proof}Suppose $r\simeq r^x$.  This means that there is some endomorphism $U$ with $r^x(h)=Ur(h)U^{-1}$ for each $h\in H$; we can clearly set $r_0(x)=U$ and $r_0(h)=r(h)$, giving a representation of $G$.  Conversely, if $r$ extends to $r_0$, $r_0(x)r(h)r_0(x)^{-1}=r^x(h)$ shows that $r\simeq r^x$.  If these two conditions hold, Frobenius Reciprocity shows that there can be at most two distinct representations that restrict to $r$ on $H$, and we have found two already, $r_0$ and $r_0\tensor\chi$.

Now given any irreducible representation $s$ of $G$, either $s|_H$ is irreducible or not.  If so we're in the case above; if not, say $s_1$ is a subrepresentation of $s|_H$.  Then by the universal property of $\Ind$, since we have an $H$-equivariant map from $s_1$ into $s$, there must be a $G$-equivariant map $\Ind^G_H s_1\rightarrow s$; by Schur's lemma and counting dimensions, we must have $\Ind^G_H s_1=s$.\end{proof}

We apply this theorem with $G=ST_A=\langle \SU(2)\times\SU(2),J\rangle$ and $H=\SU(2)\times\SU(2)$.  Given the representation $r_{k, l}$ we choose $x=J$ and find that \[J(A, B)J^{-1}=(-J_0BJ_0, -J_0AJ_0)=(J_0BJ_0^{-1},J_0AJ_0^{-1})=(J_0, J_0)(B, A)(J_0, J_0)^{-1}\]where $J_0=\begin{pmatrix}0&1\\ -1&0\end{pmatrix}$ so that $J=\begin{pmatrix}0&J_0\\-J_0&0\end{pmatrix}$.  Because $\begin{pmatrix}J_0&0\\ 0&J_0\end{pmatrix}\in\SU(2)\times\SU(2)$, we find that $r^J_{k, l}\simeq r_{l, k}$.  The representations $r_{k, l}$ are nonisomorphic for distinct pairs $(k, l)$ so the representation $r_{k, l}$ extends only for $k=l$, say to $r^1_k$ and $r^2_k$; otherwise we obtain only the induced representation, which makes no distinction between $(k, l)$ and $(l, k)$.  Hence all irreducible representations of $ST_A$ are \[r^1_k\text{ and }r^2_k\text{ for }k\geq 0\text{ and }\Ind_{ST^0_A}^{ST_A}r_{k, l}\text{ for }k>l\geq 0.\]

As discussed above and by \cite[Proposition~2.17]{MR2982436}, because $ST_A$ has two components, the field $L$ over which all endomorphisms are defined, $\End_{\overline{\Q}}(A)=\End_L(A)$, is a quadratic extension of $F$, and $ST_{A_L}$, the Sato-Tate group of $A$ as a variety over $L$, is just the identity connected component $ST^0_A=\SU(2)\times\SU(2)$ of $ST_A$.

\begin{thm}If $L$ is either a totally real field or a CM field, then Conjecture 2.2 is true for $A$ over $F$.\end{thm}
\begin{proof}If $L$ is a totally real field, this was proven already in \cite[Proposition 24]{MR3660222}, so suppose $L$ is a CM field; we proceed in a similar fashion.  We must show that for each representation given above, the $L$-function in Theorem 2.3 is invertible at $1$.  Let us first look at a representation $\Ind_{ST^0_A}^{ST_A}r_{k, l}$.  It follows from a theorem of Artin that if $s'(v')$ denotes the normalized image of Frobenius for prime $v'$ in $G_L$, then\[L^S(\Ind_{ST^0_A}^{ST_A}r_{k, l}, s)=\prod_{v\not\in S}\frac{1}{\det(1-\Ind_{ST^0_A}^{ST_A}r_{k, l}(s(v))q_v^{-s})}=\prod_{v'\not\in S'}\frac{1}{\det(1-r_{k, l}(s'(v'))q_{v'}^{-s})}=L^{S'}(r_{k, l}, s)\]so that we may prove invertibility of this new $L$-function.

From here, we cease mention of $F$ and work solely with $L$.  Let us extend $r_{k,l}$ from a representation of $\SU(2)\times\SU(2)$ to a representation $R_{k, l}$ of $G(L)$, the algebraic group coming from $G_L$ instead of $G_F$; we naturally do this by restricting $\Sym^k(St)\tensor\Sym^l(St)$ from $\GL(2)\times \GL(2)$ to $G(L)$.  In fact, we get a representation of $G^{\text{Zar}}_\ell(L)\subseteq \GL_2(\overline{\Q_\ell})\times \GL_2(\overline{\Q_\ell})$, which we can also call $R_{k, l}$.  Thus finally we get a representation of $G_L$, namely $R_{k, l}\circ \rho_{A_L, \ell}$.  Looking at where $\Frob_{v'}$ is sent, the $L$-function is\[L^{S'}(r_{k, l}, s)=L^{S'}(R_{k, l}\circ \rho_{A_L, \ell}, s+(k+l)/2)=\prod_{v'\not\in S'}\det(1-R_{k, l}\circ\rho_{A_L, \ell}(\Frob_{v'})q_{v'}^{-(s+(k+l)/2)})^{-1}.\]As discussed before the statement of Theorem 3.2, the two embeddings $\lambda_1, \lambda_2$ of $K=\End^0_L(A)$ into $\overline{\Q_\ell}$ give the decomposition of $\rho_{A_L, \ell}$ into $\rho_{A_L, \lambda_1}\oplus\rho_{A_L, \lambda_2}$, and these give the further decomposition of the $L$-function into \[L^{S'}(\Sym^k(\rho_{A_L,\lambda_1})\otimes \Sym^l(\rho_{A_L, \lambda_2}),s+(k+l)/2);\] this is finally what we must prove to be holomorphic and invertible.

We look at the weakly compatible system $(\rho_{A_L, \lambda})_{\lambda}$.  The Hodge-Tate weights of these are all $0$ and $1$.  Since the image of $\rho_{A_L, \lambda}$ is open in $G^{\text{Zar}}_{\lambda}=\GL_2(\overline{\Q_p})$, there is no subgroup of $G_L$ for which $\rho_{A_L, \lambda}$ becomes reducible.  So we may apply Theorem 2.5 to get some CM field $L'_m$ over which the compatible system $(\Sym^m(\rho_{A_L, \lambda}))_{\lambda}$ is automorphic.

The theory of cyclic base change in \cite{MR1007299} shows that $(\Sym^m(\rho_{A_L, \lambda}))_\lambda$ is automorphic over all $E$ where $L'_m/E$ is cyclic, and hence solvable; we can apply the Rankin-Selberg method as in the proof of \cite[Theorem~5.3]{MR2641185} to the field $L'=L'_kL'_l$, over which the two compatible systems $(\Sym^k(\rho_{A_L, \lambda}))_\lambda$ and $(\Sym^l(\rho_{A_L, \lambda}))_\lambda$ are both automorphic, to show that \[L^{S'}(\Sym^k(\rho_{A_L,\lambda_1}|_{G_E})\otimes \Sym^l(\rho_{A_L, \lambda_2}|_{G_E}),s+(k+l)/2)\] is invertible along the central line, assuming that $\Sym^k(\rho_{A_L,\lambda_1}|_{G_E})$ and $\Sym^l(\rho_{A_L, \lambda_2}|_{G_E})$ are not dual.  But $k\neq l$, so a dimension count shows that they cannot be dual.  So \[L^{S'}(\Sym^k(\rho_{A_L,\lambda_1}|_{G_E})\otimes \Sym^l(\rho_{A_L, \lambda_2}|_{G_E}),s+(k+l)/2)\] is invertible for all $E$ solvable subfields of $L'$; Brauer's theorem applies to the Galois groups $\Gal(L'/E)\subseteq\Gal(L'/L)$, and we get that the $L$-function for the representation over $L$ is an integer power combination of those over $E$, and therefore is also invertible.

Next, we look at the representations $r^i_k$ for $i=1, 2$ and $k\geq 1$.  Recall that they are the two distinct extensions of $\Sym^k\tensor\Sym^k$ to representations of $N(\SU(2)\times\SU(2))=\langle\SU(2)\times\SU(2), J\rangle$.  As before, let us extend $r^i_k$ to an algebraic representation of $G\subseteq\langle\GL(2)\times\GL(2), J\rangle$ by restricting $\Sym^k\tensor\Sym^k$ and leaving the image of $J$ alone.  This again gives us a representation $R_k^i$ of $G^{\text{Zar}}_{\ell}$, and then composing with $\rho_{A, \ell}$ finally gives us a Galois representation.  The $L$-function attached to $r^i_k$ is\[L^S(r^i_k, s)=\prod_{v\not\in S}\frac{1}{\det(1-r^i_k(s(v))q_v^{-s})}=\prod_{v\not\in S}\frac{1}{\det(1-R^i_k\circ\rho_{A, \ell}(\Frob_v)q_v^{-(s+k)})}\]This $L$-function being invertible follows if the $L$-functions for $R^i_k\circ\rho_{A, \ell}|_{G_E}$ for $L'/E$ solvable are, where $L'=L'_k$ is the field from Theorem 2.5.  For a given $E$, either $L\subseteq E$ or $L\not\subseteq E$.  If $L\subseteq E$, then $R^i_k\circ\rho_{A, \ell}|_{G_E}=\Sym^k(\rho_{A, \lambda_1}|_{G_E})\tensor\Sym^k(\rho_{A, \lambda_2}|_{G_E})$ as before.  Then we can apply Rankin-Selberg, except dimension count doesn't work.  We want\[L(\Sym^k(\rho_{A, \lambda_1}|_{G_E})\tensor\Sym^k(\rho_{A, \lambda_2}|_{G_E}), s+k)=L(\Sym^k(\rho_{A, \lambda_1}|_{G_E})\tensor\Sym^k(\rho_{A, \lambda_2}|_{G_E})\tensor\chi_{\ell}^{-k}, s)\] to be invertible, so we require that $\Sym^k(\rho_{A, \lambda_1}|_{G_E})$ and $\Sym^k(\rho_{A, \lambda_2}|_{G_E})\tensor\chi_{\ell}^{-k}$ not be dual.  But $\rho_{A, \lambda_1}|_{G_E}$ is essentially self-dual via the Weil pairing; in fact, $\rho_{A, \lambda_1}|_{G_E}\simeq \rho_{A, \lambda_1}^{\vee}|_{G_E}\tensor\chi_{\ell}$.  Therefore, we require that $\Sym^k(\rho_{A, \lambda_2}|_{G_E})\tensor\chi_{\ell}^{-k}$ not be isomorphic to $\Sym^k(\rho_{A, \lambda_1}|_{G_E})\tensor\chi_{\ell}^{-k}$.  But if this happened, then $\rho_{A, \lambda_2}|_{G_{E'}}\simeq\rho_{A, \lambda_1}|_{G_{E'}}$ for some finite extension $E'$.  This contradicts the fact that $\End_{\overline{\Q}}(A)=K$ by Faltings' theorem, so we're done in this case.

Otherwise, $L\not\subseteq E$, and $E$ is therefore a totally real subfield of $L'$.  But if $L=F(\sqrt{\alpha})$, then let $E'=E(\sqrt{\alpha})$ to get a degree $2$ CM extension containing $L$.  $(\Sym^k(\rho_{A, \lambda}|_{G_{E'}}))_{\lambda}$ is cuspidal automorphic as before, and the $L$-function of the $G_E$ representation is just the Asai $L$-function of the associated automorphic representation of this system, in the terminology of \cite{MR3366033}.  By \cite[Theorem~4.3]{MR3366033}, this Asai $L$-function is nonzero and holomorphic on the right half-plane, if the automorphic representation is not self-dual.  In fact, it's always nonzero, so it's holomorphic for both $r^1_k$ and $r^2_k$ if and only if the product of the two Asai $L$-functions is holomorphic.  But the product is\[L(r^1_k|_{G_E}, s)L(r^2_k|_{G_E}, s)=L(\Sym^k(\rho_{A, \lambda_1}|_{G_{E'}})\tensor\Sym^k(\rho_{A, \lambda_2}|_{G_{E'}}),s+k),\] which as before is holomorphic.  So each of these two Asai $L$-functions is holomorphic.

Finally, we look at the nontrivial finite representation $r^2_0$.  This takes $J$ to $-1$ and the connected component of the identity $ST^0_A$ to $1$.  But the $L$-function is \[\prod_{v\not\in S}\frac{1}{1-\chi(\Frob_v)q_v^{-s}},\] where $\chi$ is the Hecke character coming from $\Gal(L/F)$, and this is hence its $L$-function.  It's thus clear that this $L$-function is invertible.  So we've shown that, for every representation, the $L$-function is invertible along the line $\Re s=1$, so we're done.
\end{proof}
\begin{rmk}Notice that this proves the Sato-Tate conjecture in this case when $F=\mathbb{Q}$ because all quadratic extensions are either totally real or CM.\end{rmk}
\subsection{$\textbf{E}[D_{2n}]$, $n=2, 3, 4, 6$}
We look now at the Sato-Tate groups $ST_A=\left\langle\begin{pmatrix}B&\, \\ \,&\overline{B}\end{pmatrix}_{B\in\SU(2)},E_n:=\begin{pmatrix}e^{\frac{\pi i}{n}}\Id_2&\,\\ \,&e^{-\frac{\pi i}{n}}\Id_2\end{pmatrix},J\right\rangle$, with identity component $ST^0_A$ the embedded copy of $\SU(2)$ and component group $D_{2n}$.  These arise from abelian varieties $A$ whose endomorphism ring $\End^0_M(A)$ is a quaternion algebra for a large enough field extension $M/F$.  Either $A$ is potentially the sum of two elliptic curves without CM whose $\ell$-adic representations are twists of each other by a finite-order character, or $A$ is simple with quaternionic multiplication.  If we view $A$ as defined over $L$, where $G_L$ is the index-$2$ subgroup of the Galois group $G_F$ corresponding to the cyclic subgroup of the component group $D_{2n}$ under the correspondence given in \cite[Theorem 2.17]{MR2982436}, the endomorphism ring is not yet a quaternion algebra.  It is, however, a quadratic field $K$, as proven in \cite[Theorem~4.7]{MR2982436}; we note that while the statement in \cite{MR2982436} is constructed for the direct sum of elliptic curves case, there is no use of this in the proof, so we may apply it here as well.

To prove Conjecture 2.2 in this case, our strategy is to decompose the representation $\rho_{A, \ell}$ into a tensor $s\tensor \delta$ where $\delta$ is a finite-image dihedral representation and $s$ is a two-dimensional representation.  We do this by manually constructing a $2$-cocycle in a certain cohomology group that obstructs a representation lift from $G_L$ to $G_F$, then use the fact that the cohomology is $0$ to obtain a coboundary description, which allows us to lift.  Then we check that $s$ acts solely on the identity component and $\delta$ acts on the component group times $\pm\Id$, and finally use Rankin-Selberg and Theorem $2.5$ again.

As in the previous case, we may decompose the representation $\rho_{A, \ell}|_{G_L}$ into two $2$-dimensional pieces $\rho_{A, \lambda}$ and $\rho_{A, \overline{\lambda}}$ via the two embeddings of $K$ into $\overline{\Q_\ell}$, and as in the previous case, the theorem of Ribet says that $(\rho_{A, \lambda})_{\lambda\in S'}$ is a compatible system of representations.  But unlike the previous case, we get the isomorphism $\rho_{A, \lambda}\tensor\epsilon\simeq\rho_{A, \overline{\lambda}}$ for some finite-image character $\epsilon$.  We notice that $\Ind_{G_L}^{G_K}\rho_{A, \lambda}=\rho_{A, \ell}$ by Frobenius reciprocity, and so $\rho_{A, \ell}|_{G_L}=\rho_{A, \lambda}\oplus\rho_{A, \lambda}^g$ for $g\in G_F\backslash G_L$; therefore, $\rho_{A, \lambda}\tensor\epsilon\simeq\rho_{A, \overline{\lambda}}\simeq\rho_{A, \lambda}^g$.  (Notationally, from here we will assume that any group element $g$ with or without subscript is in $G_F\backslash G_L$ and any group element $h$ is in $G_L$, so as to repeatedly omit this statement.)
 
Because of \cite[Proposition~2.17]{MR2982436}, we know that if $M$ is the smallest field with $\End^0_M(A)$ being the full quaternion algebra, then $\Gal(M/F)=D_{2n}$, and that $\Gal(M/L)=C_n$.  Because \[(\rho_{A, \lambda}\oplus(\rho_{A, \lambda}\tensor\epsilon))|_{G_M}=\rho_{A_M, \lambda}\oplus(\rho_{A_M, \lambda}\tensor\epsilon|_{G_M})\]has a four-dimensional real endomorphism ring only if $\epsilon|_{G_M}$ is trivial, we must have $\epsilon$ being a character of $\Gal(M/L)$.  In particular, $\epsilon(h)=1$ if $h\in G_M$.  But because of the structure of $D_{2n}$, we know that $g\in G_F\backslash G_L$ has $g^2\in G_M$.  So $\epsilon(g^2)=1$.

In addition, we know \[\rho^g_{A, \lambda}\simeq \rho_{A, \lambda}\tensor\epsilon,\text{ so }\rho_{A, \lambda}\simeq \rho^g_{A, \lambda}\tensor\epsilon^g\simeq\rho_{A, \lambda}\tensor\epsilon\tensor\epsilon^g\]and hence we conclude that $\epsilon(ghg^{-1})\epsilon(h)=1$.

We let $c$ be such that \[c(h_1, h_2)=c(g_1, h_2)=1, c(h_1, g_2)=c(g'h_1, g_2)=\epsilon(h_1)\]for all $g_1, g_2, h_1, h_2$, and fixed $g'\in G_F\backslash G_L$.  Then the above statements are enough to exhaustively prove that $c$ is a cocycle in $H^2(G_F, \overline{K_{\lambda}}^{\times})$ with $\overline{K_{\lambda}}^{\times}$ having the trivial action and discrete topology.  But it's a theorem of Tate that $H^2(G_F, \overline{K_{\lambda}}^{\times})$ is trivial, so this cocycle must be a coboundary.  That means there is a continuous (i.e. finite-image) cochain $\gamma: G_F\rightarrow \overline{K_{\lambda}}^{\times}$ with $c(g_1, g_2)=\frac{\gamma(g_1)\gamma(g_2)}{\gamma(g_1g_2)}$, and so on through all combinations of $g_i$ and $h_i$.

We can check via the above the following equations: \begin{align*}\gamma(\Id)&=1\\ \, \\ \gamma(g)\gamma(g^{-1})&=c(g, g^{-1})=\epsilon(g'^{-1}g)\\ \, \\ \gamma(g)\gamma(hg^{-1})&=\gamma(ghg^{-1})c(g, hg^{-1})=\gamma(ghg^{-1})\epsilon(g'^{-1}g)=\gamma(ghg^{-1})\gamma(g)\gamma(g^{-1})\\ \, \\ \gamma(h)\gamma(g^{-1})&=\gamma(hg^{-1})c(h, g^{-1})=\gamma(hg^{-1})\epsilon(h)=\gamma(ghg^{-1})\gamma(g^{-1})\epsilon(h)\end{align*}so that $\gamma(h)=\gamma(ghg^{-1})\epsilon(h)$ for every pair $(g, h)$.  Further, $\gamma$ is a character of $G_L$; from here we only remember the domain of $\gamma$ being $G_L$.  Therefore, if we let $s_{A, \lambda}=r_{A, \lambda}\tensor\gamma$, then \[s_{A, \lambda}^g=r^g_{A, \lambda}\tensor\gamma^g\simeq r_{A, \lambda}\tensor\epsilon\tensor\gamma^g\simeq r_{A, \lambda}\tensor\gamma=s_{A, \lambda}\]so that we may extend $s_{A, \lambda}$ to be a representation of $G_F$, by Theorem $3.3$, with basis $\{s_1, s_2\}$.  And there is a clear $G_L$-equivariant map $r_{A, \lambda}\rightarrow s_{A, \lambda}\tensor\Ind_{G_L}^{G_F}\gamma^{-1}$ given by sending $v$ to $v\tensor 1$;  therefore, there is a $G_F$-equivariant map $r_{A, \ell}=\Ind_{G_L}^{G_F}r_{A, \lambda}\rightarrow s_{A, \lambda}\tensor\Ind_{G_L}^{G_F}\gamma^{-1}$.  By dimension count, they must be isomorphic.  Therefore, we are able to write $r_{A, \ell}$ as $s_{A, \lambda}\tensor\delta$, where $\delta$ is finite-image with vector space having basis $\{v_1, v_2\}$, and in fact has image isomorphic to a dihedral group.  Notice that the way we devised $\gamma$, we didn't use anything about $\lambda$, and $\epsilon$ is independent of $\lambda$ by Theorem 3.2; so $\gamma$ is independent of $\lambda$ as is $V$, so since $(r_{A, \lambda})_{\lambda}$ is a weakly compatible system, so too is $(s_{A, \lambda})_{\lambda}$.

\begin{thm}If $F$ is a totally real field and $A$ is an abelian variety defined over $F$ which has Galois type $\textbf{E}[D_n]$ for $n=2, 3, 4, 6$, then the Sato-Tate conjecture holds for $A$.\end{thm}
\begin{proof}As before, we must show that for each representation $r$ of the Sato-Tate group, the $L$-function $\prod_{v\not\in S}\det(1-r(s(v))q_v^{-s})^{-1}$ is holomorphic and invertible for $\Re s\geq 1$ where $s(v)$ is the conjugacy class given by dividing the image of $\Frob_v$ by $q_v^{1/2}$.  The Sato-Tate group $ST_A$ is given by $\SU(2)\times D_{4n}/\langle(-\Id_2, E_n^n)\rangle$, so that any representation of $ST_A$ is given by a representation of $\SU(2)$ tensored with a representation of $D_{4n}$ whose signs agree on their centers.  Of course the irreducible representations of $\SU(2)$ are $\Sym^k(St)$ and there are $4$ one-dimensional and $n-1$ two-dimensional representations of $D_{4n}$.

Our goal now is to describe where $s_{A, \lambda}$ and $\delta$ send $\Frob_v$ inside $ST_A$.  As written before, the Sato-Tate group is represented as the matrices in $\left\langle\begin{pmatrix}B&\, \\ \,&\overline{B}\end{pmatrix}_{B\in\SU(2)}, \begin{pmatrix}e^{\frac{\pi i}{n}}\Id_2&\,\\ \,&e^{-\frac{\pi i}{n}}\Id_2\end{pmatrix},J\right\rangle$.  These are inside $\Sp(4)$ where the alternating form is $\left(\begin{smallmatrix}\,&\,&1&\, \\ \,&\,&\,&1\\-1&\,&\,&\, \\ \,&-1&\,&\,\end{smallmatrix}\right)$.  However, we instead view it with the alternating form $\left(\begin{smallmatrix}\,&\,&\,&-1 \\ \,&\,&1&\,\\ \,&-1&\,&\, \\1&\,&\,&\,\end{smallmatrix}\right)$.  That is, we conjugate the Sato-Tate group by $\left(\begin{smallmatrix}1&\,&\,&\, \\ \,&1&\,&\,\\ \,&\,&\,&1 \\ \,&\,&-1&\,\end{smallmatrix}\right)$ to get the new group\[\left\langle\begin{pmatrix}B&\, \\ \,&B\end{pmatrix}_{B\in\SU(2)},\begin{pmatrix}e^{\frac{\pi i}{n}}\Id_2&\,\\ \,&e^{-\frac{\pi i}{n}}\Id_2\end{pmatrix},\begin{pmatrix}\,&\Id_2\\ \Id_2&\,\end{pmatrix}\right\rangle.\]Writing it in this form, because the Zariski closure of $\SU(2)$ is $\SL(2)$, we know that $G^1$ must contain all matrices $\left(\begin{smallmatrix}A&\,\\ \,&A\end{smallmatrix}\right)$ where $A\in\SL(2)$.  But as above, the theorem of Deligne says that the scalar multiples of the identity must be in the Zariski closure of the image of $r_{A, \ell}$, so that means that $G$ must contain all matrices of the form above, where $A$ is now in $\GL(2)$.  Now $G$ is the image under $\iota$ of $G_{\ell}^{\text{Zar}}$, the Zariski closure of the image of $r_{A, \ell}$, which is the Kronecker product of the Zariski closure of the image of $s_{A, \lambda}$ with the image of $\delta$.  If we look at the closure of $\rho_{A, \ell}(\ker \delta)$, this is a finite index subgroup of $G_{\ell}^{\text{Zar}}$.  Because the connected component of the identity $G_{\ell}^{\text{Zar},0}$ is isomorphic to $\GL(2)$ and thus is Zariski irreducible, the closure of $r_{A, \ell}(\ker\delta)$ cannot be smaller than this.

But also it cannot be larger than this: it is contained in the centralizer of a $4$-dimensional vector space inside $M_4(\overline{\Q_{\ell}})$, namely $\left(\begin{smallmatrix}a\cdot\Id&b\cdot\Id\\ c\cdot\Id&d\cdot\Id\end{smallmatrix}\right)$ in the basis $s_1\tensor v_1, s_2\tensor v_1, s_1\tensor v_2, s_2\tensor v_2$, but $G_{\ell}^{\text{Zar},0}$ is already such a centralizer: it centralizes $\left(\begin{smallmatrix}a\cdot\Id&b\cdot\Id\\ c\cdot\Id&d\cdot\Id\end{smallmatrix}\right)$ in the usual basis.  Therefore the closure of $r_{A, \ell}(\ker\delta)$ is equal to this connected component $\left\{\left(\begin{smallmatrix}A&\,\\ \,&A\end{smallmatrix}\right):A\in\GL(2)\right\}$.

On the other hand, $G_F$ can act on the vector space for the representation $r_{A, \ell}$ solely through $\delta$.  The image of this representation commutes with the kernel of $\delta$ above, but as we observed, all such matrices are of the form $\left(\begin{smallmatrix}a\cdot\Id&b\cdot\Id\\ c\cdot\Id&d\cdot\Id\end{smallmatrix}\right)$.  So the image of $G_F$ acting via $\delta$ alone lands in this vector space.  In order for the image to land in $\GSp(4)$, we can calculate that either $b=c=0$ or $a=d=0$.  Recall also that its image is dihedral and irreducible, so it must essentially give some dihedral representation.  Each matrix in a $4$-dimensional finite-image representation is unitary, so each of them already appears in the Sato-Tate group.  But the only matrices of this form in the Sato-Tate group were in the group $\langle E_n, J\rangle$, so this must be the image of $G_F$ acting through $\delta$.

We have therefore shown that the image of $\delta$ is exactly $D_{4n}$, and the closure of the image of $s_{A,\lambda}$ is $\GL(2)$.  Recall from above that a representation of the Sato-Tate group is given by the tensor product of a representation of $D_{4n}$ with a representation of $\SU(2)$ with the same sign.  Given such a representation, say $\eta\tensor\Sym^k(St)$, the $L$-function is\[\prod_{v\not\in S}\det(1-\Sym^k(s(v))\tensor\eta(s(v))q_v^{-s})^{-1}=\prod_{v\not\in S}\det(1-(\Sym^k\circ\iota\circ s_{A, \lambda})(\Frob_v)\tensor(\eta\circ\delta)(\Frob_v)q_v^{-s-k/2})^{-1}.\]We may apply Theorem 2.5 to $(s_{A, \lambda})_{\lambda}$, or in fact we may even apply \cite[Theorem~5.4.1]{MR3152941} to find a field $F'/F$ for which $(s_{A, \lambda}|_{G_{F'}})_{\lambda}$ is cuspidal automorphic, assuming $k\geq 1$.  Then as before, cyclic base change tells us that $(s_{A, \lambda}|_{G_{E}})_{\lambda}$ is cuspidal automorphic where $F'/E$ is solvable so that $L(\Sym^k|_{G_E},s)$ is invertible, and then Brauer's theorem tells us that $L(\Sym^k,s)$ is invertible as well.  We know that $\eta\circ\delta$ is cuspidal automorphic already if $\eta$ is nontrivial, so $L(\eta, s)$ is invertible.  So the Rankin-Selberg method as before tells us that the $L$-function we wanted,\[L(\Sym^k\tensor\eta,s)=\prod_{v\not\in S}\det(1-(\Sym^k\circ s_{A, \lambda})(\Frob_v)\tensor(\eta\circ\delta)(\Frob_v)q_v^{-s-k/2})^{-1},\]is invertible as long as $\Sym^k$ and $\eta$ are not dual.  For $k\geq 1$ this is obvious by cardinality, and for $k=0$ and $\eta$ nontrivial, this is just the Artin $L$-function for a representation of $\Gal(L'/F)$ where $L'$ is the fixed field of the kernel of $\delta$.  Since this is a solvable group, we know the $L$-function is invertible.
\end{proof}
\section{Other asymptotics}
So far our goal has been to show that the normalized Frobenius conjugacy classes are equidistributed within the Sato-Tate group, and from this we can deduce the distributions of the normalized traces of Frobenius in the interval $[-4,4]$.  We have done this by proving that all nontrivial irreducible representations' $L$-functions are invertible.  Unfortunately, the current state of affairs does not allow this in the two cases \textbf{A} or \textbf{B}$[C_2]$, so we set our sights a little lower.  We'd like to be able to show that for some positive fraction of primes, the trace of Frobenius is positive (resp. negative), but even this is beyond our elementary methods.  A theorem of Boxer, Calegari, Gee and Pilloni helps us in this regard, as well as a theorem of Ta{\"i}bi and Gee.  Let $A$ be any abelian surface over a totally real field $F$, and suppose that for some good prime $v$, the characteristic polynomial of the normalized Frobenius $\frac{\Frob_v}{\sqrt{q_v}}$ in its compatible system of representations is \[\text{Char}_{\frac{\Frob_v}{\sqrt{q_v}}}(X)=(X-\alpha)(X-\alpha^{-1})(X-\beta)(X-\beta^{-1})=X^4-a_1X^3+a_2X^2-a_1X+1.\]We first define $a_{1, \min}$ as the number for which zero proportion of primes $v$ have $a_1<a_{1, \min}$ but for any $\epsilon>0$ a positive proportion of $v$ have $a_1<a_{1, \min}-\epsilon$.  Let us define $a_{1, \max}, a_{2, \min}$ and $a_{2, \max}$ similarly.  We'll be able to prove the following theorems:
\begin{thm}If $A/F$ is a generic abelian surface, i.e. $\End(A_{\overline{\Q}})=\Z$, $a_{1,\min}\leq-\frac{2}{3}$ and $a_{1, \max}\geq\frac{2}{3}$.\end{thm}
\begin{thm}If $A/F$ is a generic abelian surface, $a_{2, \min}\leq\frac{4}{5}$ and $a_{2, \max}\geq\frac{4}{3}$.\end{thm}
\begin{thm}If $A/F$ is an abelian surface of type \textbf{B}$[C_2]$, $a_{1, \min}\leq -2.47$ and $a_{1, \max}\geq 2.47$.\end{thm}
\begin{thm}If $A/F$ is an abelian surface of type \textbf{B}$[C_2]$, $a_{2, \min}\leq 0.43$ and $a_{2, \max}\geq 3.57$.\end{thm}
The first two theorems above are the ``best of their kind", so to speak; that is, given the $L$-functions we currently know to be invertible, there are probability distributions of $\alpha$ and $\beta$ on the unit circle for which $a_1\geq-\frac{2}{3}$, and yet the Tauberian statistics of these $L$-functions are not violated.  These will be discussed more below, and a discussion about the best possible theorems for the \textbf{B}$[C_2]$ case will follow in an appendix.
\subsection{The generic case}
Let us state the results of Boxer-Calegari-Gee-Pilloni and Gee-Ta{\"i}bi.
\begin{thm}[{\cite[Theorem~9.2.8]{BCGP}}]Let $A$ be a challenging abelian surface over a totally real field $F$.  Then $A$ is potentially modular.\end{thm}

Challenging in the above theorem just means being in case \textbf{A} or \textbf{B}$[C_2]$.

Suppose that $(\rho_{A,\ell}, V)$ is the dual of the $\ell$-adic Tate module representation of $A$.  Suppose that $v_1, v_2, v_3, v_4$ are a symplectic basis of $V$ under the Weil pairing; that is, $\langle v_1, v_2\rangle=\langle v_3, v_4\rangle=1$ and all other pairs of vectors are $0$ under the pairing.  The Weil pairing on $V$ then becomes a direct-sum split of $\wedge^2V$:\[\wedge^2V=\Q_{\ell}(1)\oplus W\]where $\Q_{\ell}(1)$ is spanned by $v_1\wedge v_2+v_3\wedge v_4$.  It is not difficult to show that if $A$ is generic, then $W$ is irreducible.
\begin{thm}[\cite{GeeTaibi}]If $\rho_{A, \ell}$ is strongly irreducible, there is a cuspidal automorphic form $\Pi$ on $\GL(5)$ corresponding to the $W$ above.\end{thm}
\begin{proof}[Sketch]Suppose that $\pi$ is the automorphic representation corresponding to $A$.  By \cite[Theorem~A]{MR1937203}, we know that $\wedge^2\pi$ is automorphic, and is the induction of the tensor product of cuspidal automorphic representations of $\GL_{n_i}$ for $\sum n_i=6$.  We know further that $\pi$ is symplectic, so we may take $n_1=1$.

It then suffices to show that $n_1 = 1$ and $n_2 = 5$. The occurrence of more than one $n_i = 1$ is ruled out by~\cite[Theorem~1.1]{MR1610812}, and the possibility that $n_1 = 1$, $n_2 = 2$, and~$n_3 = 3$ is ruled out by~\cite[Prop~4.2]{MR2767509}.  Therefore, $\Pi_2=\Pi$ is cuspidal.\end{proof}

To prove Theorems $4.1$ and $4.2$, it suffices to prove them when looking at $A/E$ where $E/F$ is any field extension.  This is for the following reasons: if a prime $v$ of $F$ splits in $E$, the Frobenius element does not change, and neither does the size of the residue field, so that the normalized trace of Frobenius is unchanged.  Also, a set of primes of $E$ of density $1$ lie above split primes of $F$, so looking at the set of primes of $E$ described in $4.1$ or $4.2$, almost all of them lie above a split prime of $F$.  So a positive proportion of the split primes of $F$, which is a positive proportion of all primes of $F$, satisfy the inequalities.

Thus after Theorem $4.5$ we may assume that $A/F$ is modular, and so $\rho_{A, \ell}$ corresponds to a cuspidal automorphic representation.  We continue to assume $F$ totally real, as this is a further allowance in \cite{BCGP}.  We also assume that we are in the generic case \textbf{A}.  Therefore, as usual we know $L(V,s)$ is holomorphic and nonzero on $\Re(s)\geq 1$ (where the $L$-function is shifted so that the critical line is $\Re(s)=\frac{1}{2}$ and all the eigenvalues have norm $1$, as in the previous section).  In addition, since $W$ corresponds to a cuspidal representation, $L(W, s)$ is also holomorphic and nonzero on the same set.  And by Rankin-Selberg, since $V\simeq V^*\tensor\Q_{\ell}(1)$ and so $V\tensor V$ contains one copy of the cyclotomic character, $L(V\tensor V, s)$ has a simple pole at $s=1$ and is holomorphic everywhere else on $\Re(s)\geq 1$ (where again the $L$-function is normalized in the standard way).  The same holds for $W$; that is, since $W$ is irreducible and essentially self-dual, $L(W\tensor W, s)$ has a simple pole at $s=1$ and is holomorphic nonzero everywhere else on the half-plane.  And finally, since $V$ and $W$ are distinct irreducible representations, $L(V\tensor W, s)$ is holomorphic nonzero everywhere on the half plane, again by Rankin-Selberg.

Now that we have these five $L$-functions and their poles at $1$, we look back at Serre.
\begin{thm}[\cite{MR1484415}]Given a Dirichlet series \[L(\rho, s)=\prod_v\frac{1}{\det(1-\rho(x_v)q_v^{-s})}\]with a pole of order $c$ at $s=1$ and holomorphic nonzero elsewhere on $\Re(s)\geq 1$, then \[\sum_{q_v\leq n}\Tr\rho(x_v)=c\left(\frac{n}{\log n}\right)+o(n/\log n).\]\end{thm}
We apply this to the five $L$-functions above, with the normalized image of $\Frob_v$ in $V$ having eigenvalues $\alpha_v, \alpha_v^{-1}, \beta_v, \beta_v^{-1}$, to get \[\sum_{q_v\leq n}(\alpha_v+\alpha_v^{-1}+\beta_v+\beta_v^{-1})=o(n/\log n)\] and four other asymptotic equations.  Combining with the statement of Serre's theorem for the trivial representation (namely, $\displaystyle\sum_{q_v\leq n}1=n/\log n+o(n/\log n)$), and letting $s_v=\alpha_v+\alpha_v^{-1}$ and $t_v=\beta_v+\beta_v^{-1}$ for convenience, we find the system\begin{align*}
\sum_{q_v\leq n}s_v+t_v&=o(n/\log n)\\
\sum_{q_v\leq n}s_vt_v+1&=o(n/\log n)\\
\sum_{q_v\leq n}s_v^2+2s_vt_v+t_v^2-1&=o(n/\log n)&\Rightarrow \sum_{q_v\leq n}s_v^2+t_v^2-3&=o(n/\log n)\\
\sum_{q_v\leq n}s_v^2t_v+s_vt_v^2+s_v+t_v&=o(n/\log n)&\Rightarrow \sum_{q_v\leq n}s_v^2t_v+s_vt_v^2&=o(n/\log n)\\
\sum_{q_v\leq n}s_v^2t_v^2+2s_vt_v&=o(n/\log n)&\Rightarrow \sum_{q_v\leq n}s_v^2t_v^2-2&=o(n/\log n)\end{align*}

\begin{proof}[Proof of Theorem 4.1]The identity \[(2-s)(2-t)(3s+3t+2-\varepsilon)=(8-4\varepsilon)+(8+2\varepsilon)(s+t)-6(s^2+t^2)-(10+\varepsilon)st+3(s^2t+st^2)\] holds, so \begin{align*}&\, \sum_{q_v\leq n}(2-s_v)(2-t_v)(3s_v+3t_v+2-\varepsilon)\\
&=\sum_{q_v\leq n}(8-4\varepsilon)+(8+2\varepsilon)(s_v+t_v)-6(s_v^2+t_v^2)-(10+\varepsilon)s_vt_v+3(s_v^2t_v+s_vt_v^2)\\
&=\sum_{q_v\leq n}3(s_v^2t_v+s_vt_v^2)-(10+\varepsilon)(s_vt_v+1)-6(s_v^2+t_v^2-3)+(8+2\varepsilon)(s_v+t_v)-3\varepsilon\\ &=(-3\varepsilon+o(1))\frac{n}{\log n}.\end{align*}

So if $-\frac{2}{3}<a_{1, \min}=-\frac{2}{3}+\frac{\varepsilon}{3}$, then a zero proportion of primes $v$ have $a_1=s_v+t_v<-\frac{2}{3}+\frac{\varepsilon}{3}$.  And the Weil bounds on the eigenvalues hold, meaning that the sum of the left side should be positive for large enough $n$, but the right side is negative for large enough $n$.  So it's impossible for $a_{1, \min}>-\frac{2}{3}$.  The same idea holds for $a_{1, \max}$; the asymptotics above are invariant under the transformation $(s_v, t_v)\rightarrow(-s_v, -t_v)$, so if it's impossible for most primes to have their $a_1$'s lie above $-\frac{2}{3}+\frac{\varepsilon}{3}$, then it's also impossible for most primes to have their $a_1$'s lie below $\frac{2}{3}-\frac{\varepsilon}{3}$.\end{proof}

\begin{proof}[Proof of Theorem 4.2]Similarly, the following two equations hold:\[(3st+2+\varepsilon)(st+4)=3s^2t^2+(14+\varepsilon)st+(8+4\varepsilon)\] \[(5st+6-\varepsilon)(4-st)=-5s^2t^2+(14+\varepsilon)st+(24-\varepsilon),\]so\begin{align*}\sum_{q_v\leq n}(3s_vt_v+2+\varepsilon)(s_vt_v+4)&=\sum_{q_v\leq n}(8+4\varepsilon)+(14+\varepsilon)s_vt_v+3s_v^2t_v^2\\
&=\sum_{q_v\leq n}3(s_v^2t_v^2-2)+(14+\varepsilon)(s_vt_v+1)+3\varepsilon\\ &=(3\varepsilon+o(1))\frac{n}{\log n}\end{align*}and\begin{align*}\sum_{q_v\leq n}(5s_vt_v+6-\varepsilon)(4-s_vt_v)&=\sum_{q_v\leq n}(24-\varepsilon)+(14+\varepsilon)s_vt_v-5s_v^2t_v^2\\
&=\sum_{q_v\leq n}-5(s_v^2t_v^2-2)+(14+\varepsilon)(s_vt_v+1)-5\varepsilon\\ &=(-5\varepsilon+o(1))\frac{n}{\log n}\end{align*}

If $s_vt_v\leq-\frac{2}{3}-\frac{\varepsilon}{3}$ for all but a density zero set of primes $v$, then in the first equation the left side would be negative for large $n$, but the right side is positive for large $n$, impossible.  So $s_vt_v>-\frac{2}{3}-\frac{\varepsilon}{3}$ a positive proportion of the time for every positive $\varepsilon$, and hence $a_2=2+s_vt_v>\frac{4}{3}-\frac{\varepsilon}{3}$ for a positive proportion of the time.  Thus $a_{2, \max}\geq\frac{4}{3}$.

And if $s_vt_v\geq-\frac{6}{5}+\frac{\varepsilon}{5}$ for all but a density zero set of primes $v$, then in the second equation the left side would be positive for large $n$, but the right side is negative for large $n$, impossible.  So $s_vt_v<-\frac{6}{5}+\frac{\varepsilon}{5}$ a positive proportion of the time for every positive $\varepsilon$, and hence $a_2=2+s_vt_v<\frac{4}{5}+\frac{\varepsilon}{5}$ for a positive proportion of the time.  Thus $a_{2, \min}\leq\frac{4}{5}$.\end{proof}

As stated in the introduction, these are the best possible theorems we may obtain with the asymptotics arising from Serre's method; namely, if \begin{align*}s_v=0\text{ and }t_v=2\text{ for }&\,\frac{1}{6}\text{ of all primes,}\\
s_v=-\frac{3}{2}\text{ and }t_v=2\text{ for }&\,\frac{4}{21}\text{ of all primes, and}\\
s_v=\frac{-1-\sqrt{7}}{3}\text{ and }t_v=\frac{-1+\sqrt{7}}{3}\text{ for }&\,\frac{9}{14}\text{ of all primes,}\end{align*}then\begin{align*}
\sum_{q_v\leq n}s_v+t_v&=\frac{(1+o(1))n/\log n}{6}(0+2)+\frac{(4+o(1))n/\log n}{21}\left(-\frac{3}{2}+2\right)\\ &+\frac{(9+o(1))n/\log n}{14}\left(\frac{-1-\sqrt{7}}{3}+\frac{-1+\sqrt{7}}{3}\right)=\left(\frac{2}{6}+\frac{2}{21}-\frac{6}{14}+o(1)\right)\frac{n}{\log n}=o\left(\frac{n}{\log n}\right)\\ \, \end{align*}and similar equalities hold for the other four asymptotics as well.  Because $a_{1, v}$ can only ever be $-\frac{2}{3}$, $\frac{1}{2}$ or $2$, $a_{1, \min}$ is $-\frac{2}{3}$, and we cannot prove anything stronger.

A mirror equality case holds in calculating $a_{1, \max}$, and similar equality cases hold in the cases of $a_{2, \min}$ and $a_{2, \max}$.  If\begin{align*}s_v=-2\text{ and }t_v=2\text{ for }&\,\frac{1}{10}\text{ of all primes,}\\
s_v=-\frac{1}{3}\text{ and }t_v=2\text{ for }&\,\frac{9}{35}\text{ of all primes, and}\\
s_v=\frac{-1-\sqrt{7}}{3}\text{ and }t_v=\frac{-1+\sqrt{7}}{3}\text{ for }&\,\frac{9}{14}\text{ of all primes,}\end{align*}then the equalities all hold as above, and $a_{2, \max}=\frac{4}{3}$ for this set.  And if \begin{align*}s_v=2\text{ and }t_v=2\text{ for }&\,\frac{1}{52}\text{ of all primes,}\\
s_v=-2\text{ and }t_v=-2\text{ for }&\,\frac{1}{52}\text{ of all primes,}\\
s_v=-\frac{3}{5}\text{ and }t_v=2\text{ for }&\,\frac{125}{767}\text{ of all primes, and}\\
s_v=\frac{-5-\sqrt{1495}}{35}\text{ and }t_v=\frac{-5+\sqrt{1495}}{35}\text{ for }&\,\frac{1225}{1534}\text{ of all primes,}\end{align*}it is not difficult to again check that all asymptotics above hold, and $a_{2, \min}=\frac{4}{5}$ for this set.

Therefore, with our current knowledge of modularity lifting theorems, we cannot say more than these theorems.
\begin{rmk}While Theorems 4.1 and 4.2 do the job of bounding $a_{1, \min}$, etc., from above or below, they are rather weak.  We expect $a_{1, \min}$ to be equal to $-4$, yet we can only currently show that $a_{1, \min}\leq-\frac{2}{3}$, and similarly for $a_{1, \max}$.  We also expect $a_{2, \max}=6$, but we can only show that $a_{2, \max}\geq\frac{4}{3}$; and we expect $a_{2, \min}=-2$, but we can only show that $a_{2, \min}\leq\frac{4}{5}$.

Notice also that we used heavily the fact that \textbf{A} was generic, because if it were not, neither the 4-dimensional representation $V$ nor the 5-dimensional representation $W$ would need be irreducible.  Because we know the Sato-Tate conjecture in all cases except \textbf{A} and \textbf{B}$[C_2]$, we can calculate $a_{1, \min/\max}$ and $a_{2, \min/\max}$ for abelian surfaces of these types; for any abelian surface in cases \textbf{E} or \textbf{F}, where the normalized eigenvalues of Frobenius are always $2$ copies of $\alpha$ and $2$ copies of $\alpha^{-1}$, $a_{2, \max}$ is still $6$ as expected, but $a_2=4+\alpha^2+\alpha^{-2}$, so we expect (and deduce) that $a_{2, \min}=2$, so Theorem 4.2 doesn't hold if our abelian surface is not generic.\end{rmk}

\subsection{The case \textbf{B}$[C_2]$}
We now suppose our abelian variety $A$ over totally real field $F$ has Sato-Tate group $\langle\SU(2)\times\SU(2),J\rangle$.  We may still apply Theorem $4.5$, so that $A$ is potentially modular.  We of course base change to a totally real field extension $F'$ where $A$ is modular and the Tate module representation is cuspidal.  Then, as before, the representation $\rho_{A, \ell}$ is induced from a representation $\rho_{A_L, \lambda}$.  This means that $\rho_{A, \ell}\simeq \rho_{A, \ell}\tensor\chi_{L/K}$.  On the level of automorphic representations, this means that the cuspidal representation $\Pi$ coming from $\rho$ also satisfies $\Pi\simeq\Pi\tensor\chi_{L/K}$.  But this means that $\Pi$ is the base change of some cuspidal representation $\pi$ of $\GL(2)$ over $L$.

This representation $\pi$ arises from the compatible system of representations $(\rho_{A, \lambda})_{\lambda}$, and since these have big image because we're in case \textbf{B}[$C_2$], we know that the representations $\rho_{A, \lambda}$, and more generally $\Sym^k\rho_{A, \lambda}$ for any $k\geq 1$, are not induced from any character.  This means that $\Sym^k\rho_{A, \lambda}\not\simeq\Sym^k\rho_{A, \lambda}\tensor\chi$ for any character $\chi$.  We recall theorems of Kim-Shahidi:

\begin{thm}[\cite{MR1890650} Theorem 2.2.2]Let $\pi$ be a cuspidal automorphic representation of $\GL(2, \mathbb{A}_L)$, let $\omega_{\pi}$ denote the central character, and let $A^i(\pi)=\Sym^i(\pi)\tensor\omega_{\pi}^{-1}$.  Then $A^3(\pi)$ is not cuspidal if and only if there exists a nontrivial gr{\"o}ssencharacter $\mu$ such that $A^2(\pi)\simeq A^2(\pi)\tensor\mu$.\end{thm}
\begin{thm}[\cite{MR1890650} Theorem 3.3.7]With notation as above, $A^4(\pi)$ is a cuspidal representation of $\GL(5, \mathbb{A}_L)$ unless\begin{enumerate}[(1)]\item There is some nontrivial gr{\"o}ssencharacter $\eta$ with $\pi\tensor\eta\simeq\pi$
\item $A^3(\pi)$ is not cuspidal
\item $A^3(\pi)$ is cuspidal, but there is some nontrivial quadratic gr{\"o}ssencharacter $\eta$ with\\ $A^3(\pi)\simeq A^3(\pi)\tensor\eta$\end{enumerate}\end{thm}

Therefore, $A^2(\pi), A^3(\pi)$ and $A^4(\pi)$ are all automorphic.  And because $\Sym^k\rho_{A, \lambda}$ is not isomorphic to its own twist, neither is $\Sym^k\pi$.  So we obtain that $A^2(\pi), A^3(\pi)$ and $A^4(\pi)$ are cuspidal.

In the same way as above, if $\alpha_v, \alpha_v^{-1}$ are the eigenvalues of $\rho_{A_L, \lambda}(\Frob_v)q_v^{-1/2}$ for primes $v$ of $L$, and $\beta_v$, $\beta_v^{-1}$ are the eigenvalues of $\rho_{A_L, \overline{\lambda}}(\Frob_v)q_v^{-1/2}$, and for simplicity we denote $x_v=\alpha_v+\alpha_v^{-1}$ and $y_v=\beta_v+\beta_v^{-1}$, then via Rankin-Selberg we find that if $0\leq k, l\leq 4$ or if one of $k, l$ equals $0$ and the other is at most $8$, then \[\sum_{q_v<n}x_v^ky_v^l=\begin{cases}(C_{k/2}C_{l/2}+o(1))\frac{n}{\ln n},&k, l\text{ both even}\\ \frac{o(1)n}{\ln n},&\text{one of }k, l\text{ odd}\end{cases}\]where $C_n=\frac{1}{n+1}\binom{2n}{n}$ is the $n$'th Catalan number.

\begin{proof}[Proof of Theorem 4.3]Let\begin{align*}Q(x, y)&=-12.543(x+y)+53.838(x^2+y^2)-12.954(x^3+y^3)-13.063(x^4+y^4)-7.914(x^5+y^5)\\
&-2.9(x^6+y^6)+3.607(x^7+y^7)+1.575(x^8+y^8)+124.68xy-183.789(x^2y+y^2x)\\
&+1.878(x^3y+y^3x)+50.255(x^4y+y^4x)+117.628x^2y^2+73.149(x^3y^2+y^3x^2)\\
&-48.646(x^4y^2+y^4x^2)-65.928x^3y^3+8.734(x^4y^3+y^4x^3)+1.098x^4y^4\end{align*}(All decimals are exact, unless otherwise noted.)  It's easy to check that the minimum of $Q(x, y)$ on the set $\{x, y\in [-2, 2]:x+y\geq -2.47\}$ is when $x\approx-1.81913$ and $y\approx0.644208$, giving a minimum of approximately $-1.93656$, and yet the sum \[\sum_{q_v<n}Q(x_v, y_v)=\frac{(-2.04+o(1))n}{\ln n}.\]So it is impossible for $x_v+y_v$ to always be $\geq -2.47$, and therefore $a_{1, \min}\leq -2.47$.  And each asymptotic above is invariant under $(x, y)\rightarrow (-x, -y)$, so a mirror polynomial proves that $a_{1, \max}\geq 2.47$.\end{proof}
\begin{proof}[Proof of Theorem 4.4]Let \begin{align*}R(x, y)&=-24.04(x^2+y^2)+39.64(x^4+y^4)-13.14(x^6+y^6)+3.82(x^8+y^8)-15.76xy\\&-119.88(x^3y+y^3x)+484.32x^2y^2-153.28(x^4y^2+y^4x^2)+192.44x^3y^3+8.2x^4y^4\end{align*}It's easy to check that the minimum of $R(x, y)$ on the set $\{x, y\in[-2, 2]:xy\geq -1.57\}$ is when $x\approx0.907648$ and $y\approx0.188967$, for a minimum of approximately $-8.32369$, and yet the sum\[\sum_{q_v<n}R(x_v, y_v)=\frac{(-9.96+o(1))n}{\ln n}.\]So it is impossible for $x_vy_v$ to always be $\geq -1.57$, and therefore $a_{2, \min}\leq -1.57+2=0.43$.  And each asymptotic above is invariant under $(x, y)\rightarrow (-x, y)$, so a mirror polynomial proves that $a_{2, \max}\geq 3.57$.\end{proof}
\newpage\section{Appendix}
The polynomials we used to prove Theorems $4.3$ and $4.4$ appear rather arbitrary; besides the fact that they work, they give no indication of how strong the results are, how tight the bounds of $2.47$ and $1.57$ are.  There are two questions this appendix answers: the tightness of these bounds (in a similar manner to how we showed $4.1$ and $4.2$ gave the best known bounds in the generic case), and the method used to derive them.

Let $V$ be the set $\{(x, y)\in[-2, 2]\times[-2,2]:x+y\geq u\}$ or $\{(x, y)\in[-2, 2]\times[-2,2]:xy\geq v\}$ for some $u$ or $v$, and let $f:[-2, 2]\times[-2, 2]\rightarrow\mathbb{R}^{32}$ via \[(x, y)\rightarrow(x, y, x^2, xy, y^2, \ldots x^8, x^4y^4, y^8).\]  Either the convex hull of $f(V)$ contains $O=(0, 0, 1, 0, 1, 0, 0, \ldots 14, 4, 14)$, or it does not.  If $O$ is contained in the convex hull of $f(V)$, by Caratheodory's theorem, it can be written as the convex combination of $33$ points in the image $f(V)$.  These points give us pairs $(x, y)$ and coefficients, or probabilities, which satisfy the asymptotics that we derived.  So we would be unable to prove that $a_{1, \min}$ or $a_{2, \min}$ was any smaller than $u$ or $v+2$.

On the other hand, if $O$ is not contained in the convex hull of $f(V)$, there is a hyperplane separating $O$ from this convex hull.  Namely, there is some linear combination of the $32$ coordinates which is smaller than some constant $c$ for $O$, and larger than $c$ for every point in $f(V)$.  This hyperplane gives us a polynomial with which we may prove upper bounds for $a_{1, \min}$ and $a_{2, \min}$, as we did when we proved Theorems $4.1$ through $4.4$.

Because increasing $u$ or $v$ only shrinks $V$, the sets of $u$ and $v$ for which $O$ is contained in the convex hull form intervals, as do the sets of $u$ and $v$ where $O$ is not contained.  Therefore, the supremum of the former is the infimum of the latter, and for that $u$ or that $v$, we obtain both a proof and an example, and this is the best we can hope for.  In the case of the generic abelian surface, the upper and lower bounds we obtained were easy rational numbers, but there's no reason to suspect this to be the case for a \textbf{B}$[C_2]$ surface.  The rest of this appendix is devoted to finding tight provable bounds on the suprema. 

We first prove that $a_{1, \min}$ will be less than or equal to $-2.4763827913319\ldots$.  Look at the polynomial $P_1(x, y)$, symmetric in $x$ and $y$, with the following (exact) coefficients (unfilled for $x^iy^j$ where $j>i$):

$\begin{array}{c|ccccc}\,&1&y&y^2&y^3&y^4\\ \hline
1&0&\,&\,&\,&\, \\
x&-9.6430622783853&108.9702541224326&\,&\,&\, \\
x^2&49.2216326267277&-180.0171980017891&125.0609266454326&\,&\, \\
x^3&-9.225013979636&6.9445854923998&68.1838852970187&-66.0585984730189&\, \\
x^4&-11.7940568488902&49.3497768306&-48.7776655621495&9.217112694634&1\\
x^5&-10.4048835085938&\,&\,&\,&\, \\
x^6&-3.4018229998967&\,&\,&\,&\, \\
x^7&4.1057063608821&\,&\,&\,&\, \\
x^8&1.7252053549918&\,&\,&\,&\, \end{array}$
The minimum of $P_1(x, y)$ on $\{(x, y)\in[-2, 2]\times[-2,2]:x+y\geq-2.4763827913319\}$ is approximately $-0.495177804465548$, at $x=y\approx 1.122946224307864$.  However, $\displaystyle\sum_{q_v<n}P_1(x_v, y_v)=(-0.4951778044674+o(1))\frac{n}{\ln n}$.

On the other hand, there is a set of $33$ points $(x, y)$ within $[-2, 2]\times [-2, 2]$ whose coordinate sum is always at least $-2.4763827913320$, plotted below, for which the $32$-vector\[(0, 0, 1, 0, 1, 0, 0, 0, 0, 2, 0, 1, 0, 2, 0, 0, 0, 0, 0, 0, 5, 2, 0, 2, 5, 0, 0, 0, 0, 14, 4, 14)\]is inside the convex hull of the points given by \begin{align*}(x, &y, x^2, x y, y^2, x^3, x^2 y, x y^2, y^3, x^4, x^3 y, x^2 y^2, x y^3, y^4, x^5, x^4 y, x^3 y^2,\\
x^2 y^3,& x y^4, y^5, x^6, x^4 y^2, x^3 y^3, x^2 y^4, y^6, x^7, x^4 y^3, x^3 y^4, y^7, x^8, x^4 y^4, y^8):\end{align*}
\begin{figure}[h]
	\centering
		\includegraphics[scale=0.20]{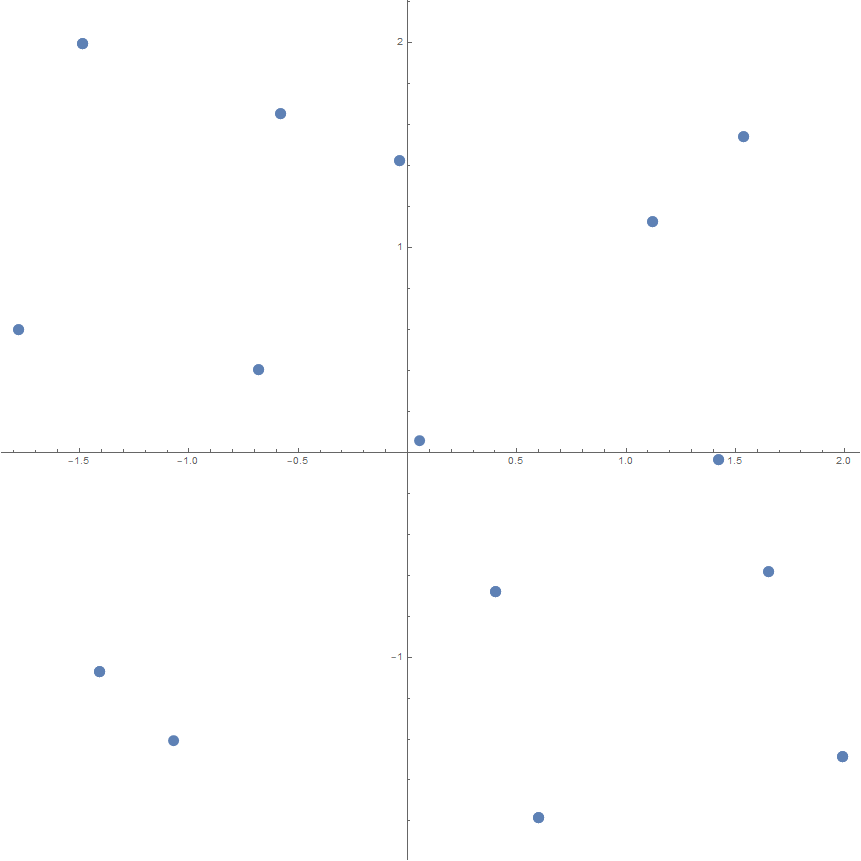}
\end{figure}

\begin{align*}(0.40233388785758, -0.68162727157206)&, (-0.68162490825764, 0.40233317377632), \\
(-0.03593446385013, 1.4223373527278)&, (-0.58181793464029, 1.65045045907013), \\
(0.59759350821447, -1.78077844166752)&, (1.53829446803677, 1.53829443533382), \\
(-1.48621983094263, 1.99140650840038)&, (1.42233731135369, -0.03593490350603), \\
(0.05438775487699, 0.05438886977203)&, (-1.78077900893326, 0.59759354704086), \\
(1.99140617335252, -1.4862186561741)&, (-1.40798021804983, -1.06840257328206), \\
(0.59759347905152, -1.78077910495742)&, (-1.48621965507992, 1.99140661320125), \\
(1.12294676572784, 1.12294624842174)&, (1.42233869303903, -0.03593747650892), \\
(1.65045062298356, -0.58181827821828)&, (0.40233336753712, -0.6816247052117), \\
(-0.68162690245729, 0.40233386944888)&, (0.40233322519093, -0.681625459605), \\
(1.12294556005286, 1.12294583096732)&, (-0.03593985561373, 1.42233955543634), \\
(-0.58181951983038, 1.65045131684197)&, (1.99140650540332, -1.48621961266743), \\
(1.65045237751894, -0.58182231464692)&, (-1.78077877657724, 0.59759334939547), \\
(-1.06840211927449, -1.4079806720575)&, (1.53829124404177, 1.53829176844063), \\
(-1.40797805061751, -1.06840474071448)&, (1.9914061436234, -1.48622121985896), \\
(0.59759450830252, -1.7807799162483)&, (-1.4862239646326, 1.99140795748714), \\
(-1.40798623761224, -1.06839655371975)&\end{align*}

(Again, all coordinates exact.)  So in fact the best we can prove here is just that $a_{1, \min}\leq -2.4763827913319$, and that $a_{1, \max}\geq 2.4763827913319$ with a mirror polynomial and points.

Similarly, let $P_2(x, y)$ denote the polynomial with (exact) coefficients below which is symmetric in $x$ and $y$:

$\begin{array}{c|ccccc}\,&1&y&y^2&y^3&y^4\\ \hline
1&0&\,&\,&\,&\, \\
x&0&-0.9148488345531369&\,&\,&\, \\
x^2&-2.0489539067392863&0&44.9702636684728257&\,&\, \\
x^3&0&-10.7425748658745577&0&16.8692193520802346&\, \\
x^4& 3.6839213331709682&0&-14.3548663298347627&0&1\\
x^5&0&\,&\,&\,&\, \\
x^6&-1.4264194026272393&\,&\,&\,&\, \\
x^7&0&\,&\,&\,&\, \\
x^8&0.4106920221952855&\,&\,&\,&\, \end{array}$

The minimum of $P_2(x, y)$ on the set $\{(x, y)\in[-2, 2]\times[-2,2]:xy\geq -1.578548220646049\}$ is at $x\approx -1.647233715535326, y\approx -0.553436099672013$, with a value of approximately $-0.576241536465307$, but $\displaystyle\sum_{q_v<n}P_2(x_v, y_v)=(-0.5762415364653239+o(1))\frac{n}{\ln n}$.

But the set of points $(x, y)$ listed (exactly) and plotted below also has the $32$-vector\[(0, 0, 1, 0, 1, 0, 0, 0, 0, 2, 0, 1, 0, 2, 0, 0, 0, 0, 0, 0, 5, 2, 0, 2, 5, 0, 0, 0, 0, 14, 4, 14)\]inside the convex hull of the points given by \begin{align*}(x, &y, x^2, x y, y^2, x^3, x^2 y, x y^2, y^3, x^4, x^3 y, x^2 y^2, x y^3, y^4, x^5, x^4 y, x^3 y^2,\\
x^2 y^3,& x y^4, y^5, x^6, x^4 y^2, x^3 y^3, x^2 y^4, y^6, x^7, x^4 y^3, x^3 y^4, y^7, x^8, x^4 y^4, y^8):\end{align*}

\begin{figure}[h]
	\centering
		\includegraphics[scale=0.16]{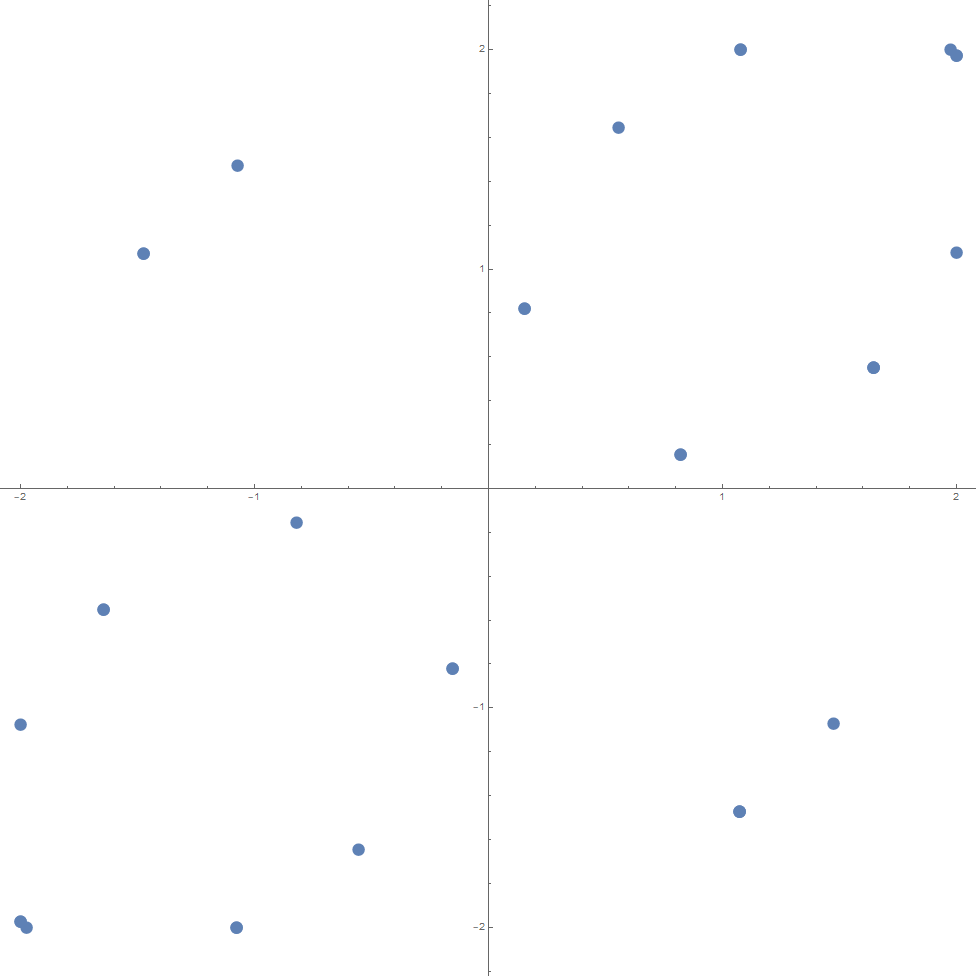}
\end{figure}

\begin{align*}
&(0.15506049352336642, 0.82103437036363329), (-1.07751316618925008, -2),\\ &
(-0.55343613654977384, -1.64723374649387681), (-0.15506048529352139, -0.82103434384587391), \\ &
(1.64723372391649941, 0.5534361137417255), (-1.9731805874505989, -2), \\ &
(0.82103437282171523, 0.15506048372780666), (1.07134858923922885, -1.47342166359409476), \\ &
(-0.82103433524791835, -0.15506047870672044), (0.15506046708915971, 0.82103431347805562), \\ &
(2, 1.07751316910812552), (-2, -1.07751316683949198), \\ &
(1.64723377312861428, 0.55343615943553201), (0.55343612726989892, 1.64723373754207403), \\ &
(1.97318058013085052, 2), (1.07751316252419741, 2), \\ &
(0.82103432514680383, 0.1550604774370247), (1.07134859824340295, -1.47342165121068827), \\ &
(-1.64723372758280026, -0.5534361192094094), (-1.64723373704703668, -0.55343612092654846), \\ &
(-0.15506049299061308, -0.82103439731133801), (-2, -1.97318058198809136), \\ &
(2, 1.9731805809913704), (1.47342165529506514, -1.07134859527358692), \\ &
(1.07751320597324756, 2), (-1.47342164800053422, 1.07134860057755775), \\ &
(-2, -1.97318060745136972), (-1.47342169316353921, 1.07134856773881076), \\ &
(-1.0713485853406953, 1.47342166895573339), (-1.0775132033440292, -2), \\ &
(1.64723371168303767, 0.55343607239179169), (2, 1.97318056276917512), \\ &
(1.07134859374131587, -1.47342165740239169)\end{align*}all of whose products are at least $-1.5785482206460513$.  And sending $x$ to $-x$ and leaving $y$ alone gives us a polynomial and points to prove the lower bound for $a_{2, \max}$.

So we've shown that the bounds $a_{1, \min}\leq -2.4763827913319$, $a_{1, \max}\geq 2.4763827913319$, $a_{2, \min}\leq 0.421451779353951$, and $a_{2, \max}\geq 3.578548220646051$ are approximately the best we can prove with the theorems of Kim-Shahidi.
\nocite{*}
\bibliographystyle{alpha}
\bibliography{LaTeX1bib}
\end{document}